\documentclass[11pt]{amsart}
\usepackage{amsmath,amssymb,amsthm}
\usepackage[latin1]{inputenc}
\usepackage{version,tabularx,multicol}
\usepackage{graphicx,float, url}
\usepackage{amsfonts}

\usepackage{graphicx}
\usepackage{epsfig}
\usepackage{subfigure}
\usepackage{float}
\usepackage{color}
\usepackage{flafter}

\newcommand{\reff}[1]{(\ref{#1})}

\vspace{0.3cm}

\theoremstyle{plain}
\newtheorem{theo}{Theorem}[section]
\newtheorem{corollary}[theo]{Corollary}
\newtheorem{cor}[theo]{Corollary}
\newtheorem{prop}[theo]{Proposition}

\newtheorem{lem}[theo]{Lemma}

\theoremstyle{remark}
\newtheorem{rem}[theo]{Remark}

\newcommand{\cp}{{\mathcal P}}

\newcommand{\cy}{{\mathcal Y}}

\newcommand{\E}{{\mathbb E}}

\renewcommand{\P}{{\mathbb P}}

\newcommand{\ind}{{\bf 1}}

\newcommand{\val}[1]{\mathop{\left| #1 \right|}\nolimits}
\newcommand{\inv}[1]{\mathop{\frac{1}{ #1}}\nolimits}
\newcommand{\expp}[1]{\mathop {\mathrm{e}^{ #1}}}

\newcommand{\ixen}{X^{(n)}}
\newcommand{\igreken}{Y^{(n)}}

\newcommand{\aen}{A^{(n)}}
\newcommand{\taen}{\tilde A^{(n)}}
\newcommand{\haen}{\hat A^{(n)}}

\newcommand{\ton}{\tau^{(n)}}
\newcommand{\eren}{R^{(n)}}
\newcommand{\teen}{T^{(n)}}

\newcommand{\sigmen}{\sigma^{(n)}}

\begin{document}

\title{On the length of an external branch in the Beta-coalescent}

\date{\today}

\author{Jean-St\'ephane Dhersin}
\address{D\'epartement de Math\'ematiques, Universit\'e Paris 13, 99 av. J-B. Cl\'ement, F-93430 Villetaneuse, France}
\email{dhersin@math.univ-paris13.fr}

 \author{Fabian Freund}
  \address{Crop Plant Biodiversity and Breeding Informatics Group (350b), Institute of Plant Breeding, Seed Science and Population Genetics,   University of Hohenheim, Fruwirthstrasse 21, 70599 Stuttgart, Germany}
 \email{ffreund@uni-hohenheim.de}

\author{Arno Siri-J{\'e}gousse}
\address{CIMAT, A.C., Guanajuato, Mexico}
\email{arno@cimat.mx}

\author{Linglong Yuan}
\address{D\'epartement de Math\'ematiques, Universit\'e Paris 13, 99 av. J-B. Cl\'ement, F-93430 Villetaneuse, France}
\email{yuan@math.univ-paris13.fr}

\begin{abstract}
 In this paper, we consider Beta$(2-{\alpha},{\alpha})$ (with $1<{\alpha}<2$) and related ${\Lambda}$-coalescents. If $T^{(n)}$ denotes the length of an external branch of the $n$-coalescent, we prove the convergence of $n^{{\alpha}-1}T^{(n)}$ when $n$ tends to $ \infty $, and give the limit. To this aim, we give asymptotics for the number $\sigma^{(n)}$ of collisions which occur in the $n$-coalescent until the end of the chosen external branch, and for the block counting process associated with the $n$-coalescent. \end{abstract}

\keywords{Coalescent process, Beta-coalescent, external branch, block counting process, recursive construction}
\subjclass[2010]{60J70, 60J80, 60J25,  60F05,  92D25}

\maketitle

\section{Introduction}
\subsection{Motivation and main results }
In modern genetics, it is possible to sequence whole genomes of individuals. In order to put this information to maximal use, it is important to have well-fitting models for the gene genealogy 
of a sample of individuals. The standard model for a gene genealogy of a sample of $n$ individuals is Kingman's $n$-coalescent (see \cite{MR671034}, \cite{MR633178}). Kingman's $n$-coalescent is a continuous-time 
Markov process with state space $\mathcal{P}^{(n)}$, the set of partitions of $\left\{ 1,\ldots,n\right\}$. The process starts in the trivial partition $(\{1\},\ldots,\{n\})$ and transitions are only possible as mergers of exactly two blocks of the current state. Each 
such binary merger occurs with rate 1. These mergers are also called collisions.

For many populations, Kingman's $n$-coalescent describes the genealogy quite well. Kingman showed in \cite{MR633178} that the ancestral trees of a sample of size $n$ in populations with size $N$ 
evolving by a Wright-Fisher model will converge weakly to Kingman's $n$-coalescent for $N\to\infty$ (after a suitable time-change). This result is relatively robust if population evolution deviates from 
the Wright-Fisher model (see \cite{MR633178} or \cite{MR1808909}). However, there is evidence that there are populations where the gene genealogy of a sample is not described well by Kingman's $n$-coalescent. Examples
of such populations can be found in maritime species, where one individual can have a huge number of offspring with non-negligible probability (see \cite{arnason}, \cite{boom1994mdv} \cite{hedgecock}, \cite{MR2365877} and \cite{Eldon2006}).

A whole class of potential models for the gene genealogy of a sample was introduced independently by Pitman and Sagitov (see \cite{MR1742892} and \cite{MR1742154}): The class of $n$-coalescents with multiple collisions. 
A $n$-coalescent 
with multiple collisions is a continuous-time Markov process with state space $\mathcal{P}^{(n)}$, where all possible transitions are done by merging two or more blocks of the current state into one 
new block. Every $n$-coalescent $\Pi^{(n)}$ is exchangeable, meaning $\tau\circ \Pi^{(n)}\stackrel{d}{=} \Pi^{(n)}$ for every permutation $\tau$ of $\left\{ 1,\ldots,n\right\}$. The transition rate of a merger/collision of $k$ of $b$ present blocks is given by
\begin{equation}\label{eq:Prates}
\lambda_{b,k}=\int_0^1 x^k(1-x)^{b-k}x^{-2} \Lambda(dx)
\end{equation}
for a finite measure $\Lambda$ on $[0,1]$ (this definition is due to Pitman \cite{MR1742892}). Since the process is characterized by the measure $\Lambda$, it is also called a $\Lambda$-$n$-coalescent. Note that Kingman's $n$-coalescent is a $\Lambda$-$n$-coalescent with $\Lambda$ being the Dirac measure $\delta_0$ in 0.\\
{It is possible to define a continuous-time process $\Pi$ with state space $\mathcal{P}$, the partitions of the natural numbers $\mathbb{N}$, whose restriction on $\{1,\ldots,n\}$ is a $\Lambda$-$n$-coalescent for all $n\in\mathbb{N}$. Such a process is called a $\Lambda$-coalescent.\\}
An important subclass of $\Lambda$-$n$-coalescents are Beta $n$-coalescents characterized by $\Lambda$ being a Beta distribution, especially for the choice of parameters $2-\alpha$ and $\alpha$ for 
$\alpha\in(0,2)$. The class of $Beta(2-\alpha,\alpha)$-$n$-coalescents appears as ancestral trees in various settings. They appear in the context of supercritical Galton-Watson processes 
(see \cite{MR1983046}), of continuous-state branching processes (see \cite{MR2120246}) and of continuous random trees (see \cite{MR2349577}). They also seem to be a class where suitable models for the ancestral 
tree can be found for samples from species who do not fit well with the Kingman-setting (see \cite{BirknerBlath2008_JMB}). Note that for $\alpha\to 2$, the rates of the $Beta(2-\alpha,\alpha)$-$n$-coalescent converge 
to the rates of Kingman's $n$-coalescent. In this sense, Kingman's $n$-coalescent can be seen as a border case of this class of Beta $n$-coalescents.

For $\alpha=1$, $Beta(2-\alpha,\alpha)$ is the uniform distribution on $[0,1]$. The corresponding $n$-coalescent is the Bolthausen-Sznitman $n$-coalescent. It appears in the field of spin glasses (see \cite{bolthausen1998ruelle}, \cite{bovier2007much}) and is also connected to random recursive trees (see \cite{MR2164028}).

Let us denote by ${\Pi}^{(n)}=(\Pi^{(n)}_t)_{t\geq 0}$ a $n$-coalescent. In this paper, we are interested in three functionals of $n$-coalescents
\begin{itemize}
\item the length $T^{(n)}$ of a randomly chosen external branch ;
\item the number $\sigma^{(n)}$ of collisions which occur in the $n$-coalescent until the end of a randomly chosen external branch ;
\item the block counting process $R^{(n)}=(R^{(n)}_t)_{t\geq 0}$: $R^{(n)}_t=|\Pi^{(n)}_t|$ is the number of blocks of $\Pi^{(n)}_t$.
\end{itemize}
 Note that $T^{(n)}$ can also be characterized as the waiting time for the first collision of a randomly chosen individual and $\sigmen$ as the number of collisions we have to wait to see the randomly chosen individual merge. For $i\in\left\{ 1,\ldots,n\right\}$ define
$$T^{(n)}_i:=\inf\left\{t|\left\{ i\right\}\notin \Pi^{(n)}_t\right\}$$
as the length of the $i$th external branch and
$$\sigma^{(n)}_i:=\inf\left\{k|\{i\}\notin \pi_k\right\}$$
as the number of collisions until the end of the $i$th external branch, where $\pi_k$ is the state of the $n$-coalescent after $k$ jumps. Due to the exchangeability of the $n$-coalescent, we have $T^{(n)}\stackrel{d}{=} T^{(n)}_1$ and $\sigma^{(n)}\stackrel{d}{=}\sigma^{(n)}_1$. Since we are only interested in distributional results, for the remainder of the article we will identify $T^{(n)}$ with $T^{(n)}_1$ and $\sigma^{(n)}$ with $\sigma^{(n)}_1$.

If the $n$-coalescent is used as a model for an ancestral tree of a sample of individuals/genes, the functionals $\teen$ and $\sigmen$ can be interpreted biologically. The length of an external branch measures the uniqueness of the individual linked to that branch compared to the sample, since it gives the time this individual has to evolve by mutations that do not affect the rest of the sample (see the introduction of \cite{caliebe2007length} for more information). It was first introduced by Fu and Li in \cite{Fu1993}, where they compare mutations on external and internal branches of Kingman's $n$-coalescent in order to test for the neutrality of mutations.

The functional $\sigma^{(n)}$ was first introduced in \cite{caliebe2007length}, though $n-\sigma^{(n)}$ was also analyzed in \cite{MR2156553} as the level of coalescence of the chosen individual with the rest of the sample. In both articles, the functionals were defined for Kingman's $n$-coalescent.

For the biological interpretation of $\sigmen$, we see the $n$-coalescent as an ancestral tree of a sample of size $n$. Each collision in the $n$-coalescent then resembles the emergence of an
ancestor of the sample. $\sigma^{(n)}-1$ is the number of ancestors
of the sample which emerge before the most recent ancestor of the randomly
chosen individual/gene emerges (recall that time runs backwards in the $n$-coalescent). In this line of thought, $\sigma^{(n)}$ gives the temporal
position of the first ancestor of the chosen individual/gene among all
ancestors of the sample, which are the $\tau^{(n)}$ collisions in the $n$-coalescent. Thus, $\frac{\sigma^{(n)}}{\tau^{(n)}}$ gives the relative temporal position of the first ancestor of the chosen individual/gene among 
all ancestors of the sample (until the most recent common ancestor). In this sense, we interpret $\frac{\sigma^{(n)}}{\tau^{(n)}}$ as a measure of how ancient the chosen
individual/gene is compared to the rest of the sample.

In this article, we focus on the asymptotics of $T^{(n)}$, $\sigma^{(n)}$ and $R^{(n)}$ for $n\to\infty$.   
The asymptotics of these functionals are already known for some $\Lambda$-$n$-coalescents. For $T^{(n)}$, we have
\begin{itemize}
\item $\Lambda=\delta_0$ (Kingman's coalescent): $nT^{(n)}\stackrel{d}{\to} T$, where $T$ has density $t\mapsto \frac{8}{(2+t)^3}$ (see \cite{MR2156553}, \cite{caliebe2007length}, \cite{Janson11theexternal}),
\item $\Lambda=Beta(1,1)$ (Bolthausen-Sznitman coalescent): $\log(n)T^{(n)}\stackrel{d}{\to} Exp(1)$ (see \cite{MR2554368}),
\item $\Lambda$ with $\mu_{-1}=\int_{0}^1 x^{-1}\Lambda(dx)<\infty$: $T^{(n)}\stackrel{d}{\to}Exp(\mu_{-1})$ (see \cite{MR2684740}, see also \cite{MR2484170})
\end{itemize}
for $n\to\infty$. For $\sigma^{(n)}$, we have
\begin{itemize}
\item $\Lambda=\delta_0$: $\sigmen/n\stackrel{d}{\to} Beta(1,2)$ (see \cite{caliebe2007length}),
\item $\Lambda=Beta(1,1)$: $\frac{\log(n)}{n}\sigmen\stackrel{d}{\to} Beta(1,1)$ (see \cite{MR2554368}),
\item $\Lambda$ with $\mu_{-2}=\int_{0}^1 x^{-2}\Lambda(dx)<\infty$: $\sigmen\stackrel{d}{\to}Geo(\frac{\mu_{-1}}{\mu_{-2}})$ (see \cite{MR2484170}) 
\end{itemize}
for $n\to\infty$, where $Geo(p)$ is the geometric distribution on $\mathbb{N}$ with parameter $p$.

{There are some known results for the asymptotics of the block counting process $R^{(n)}$. Define $R=(R_t)_{t\geq 0}$ by setting $R_t:=|\Pi_t|$, the number of blocks of a $\Lambda$-coalescent $\Pi=(\Pi_t)_{t\geq 0}$. Note that $R=\lim_{n\to\infty}R^{(n)}$, if the block counting process $R^{(n)}$ is defined for the restriction of $\Pi$ to $\{1,\ldots,n\}$ for each $n\in\mathbb{N}$. For the small-time behaviour of the block counting process $R$, the following results are known:} 
\begin{itemize}
\item If $\Lambda(dx)=f(x)dx$, where $f(x)\sim Ax^{1-\alpha}$ for $1<\alpha<2$ and $\sim$ means that the ratio of the two sides tends to one as $x\rightarrow 0+$. Then 
\begin{equation}
\label{eq:R}
\lim_{t\rightarrow 0+}t^{1/(\alpha-1)}R_t=(\frac{\alpha}{A\Gamma(2-\alpha)})^{1/(\alpha-1)}\quad a.s.
\end{equation}
(see \cite{MR2349577} and \cite{berestycki-2008-44})
\item If $\int_{[\epsilon,1]}x^{-2}{\Lambda}(dx)=\epsilon^{-\alpha}L_\epsilon$ with $\epsilon \in(0,1)$ and $L_\epsilon$ slowly varying as $\epsilon\rightarrow 0$. Denote by $g_\epsilon:=L_\epsilon^{-1}\epsilon^{\alpha-1}$, then
$$\lim_{\epsilon\rightarrow 0+}\epsilon R_{g_\epsilon}=(\Gamma(2-\alpha))^{1/(1-\alpha)},\quad \text{in Probability}.$$
(see \cite{MR2247827}).

Notice that both cases are complementary. The block counting process of Kingman coalescent gets similar limit value almost surely by taking $\alpha=2$, (see \cite{MR2349577}, page $216$).
\item For any $\Lambda$ coalescent that comes down from infinity(that means for any $t>0$, $R_t<\infty$ almost surely), there exists a deterministic positive function $v_t$ on $(0,\infty)$, such that:
$$\lim_{t\rightarrow 0+}\frac{R_t}{v_t}=1,\quad a.s.$$
For explicit form of $v$ and finer results, we refer to \cite{berestycki-2008}. This result is much more general than the two former ones.
\end{itemize}

We will analyze the asymptotics for $T^{(n)}$, $\sigma^{(n)}$ and $R^{(n)}$ for $\Lambda$-$n$-coalescents with $\Lambda$ fulfilling
$$\rho(t)=C_0t^{-\alpha}  +O(t^{-\alpha+\zeta}),\ t\rightarrow 0$$
for  some $C_0>0$, {$\alpha\in(1,2)$} and $\zeta>1-1/\alpha$, where $\rho(t)=\int^1_t x^{-2}\Lambda(dx)$. Note that this class of $n$-coalescents includes all $Beta(a,b)$-$n$-coalescents with parameters $a \in (0,1)$ and $b>0$.  In this class of $n$-coalescents, we have the following asymptotics for $T^{(n)}$, $\sigma^{(n)}$ and $R^{(n)}$:
\begin{itemize}
\item $\frac{\sigmen}{n(\alpha-1)}\xrightarrow{d}\sigma$,
\item $n^{\alpha-1}\teen\overset{d}{\to}\inv{C_0\Gamma(2-\alpha)}((1-\sigma)^{1-\alpha}-1)$,
\item for any $t_{0}>0,\varepsilon>0$, $\displaystyle \displaystyle \P(\sup_{0\leq t \leq t_{0}}|n^{-1}\eren_{tn^{1-\alpha}}-(1+C_0\Gamma(2-\alpha)t)^{-1/(\alpha-1)}|>\varepsilon)\to0$,
\end{itemize}
for $n\to\infty$ , where $\sigma\overset{d}= Beta(1,\alpha)$. 

Remark that if $\Lambda(dx)=f(x)dx$, where $f(x)\sim Ax^{1-\alpha}$ for $1<\alpha<2$, we get that for fixed $t>0$, 
{
\[
\lim_{n\to \infty }\frac{1}{n}(1+\frac{A\Gamma(2-\alpha)}{\alpha}t)^{1/(\alpha-1)}\eren_{tn^{1-\alpha}}=1
\]
in probability, which is to be compared to the following a.s. convergence implied by (\ref{eq:R}):
\[
\lim_{n\to \infty }\frac{1}{n}(\frac{A\Gamma(2-\alpha)}{\alpha}t)^{1/(\alpha-1)}R_{tn^{1-\alpha}}=1.
\]}
{The different scaling comes from the fact that we do not consider the block counting process $R=\lim_{n\to\infty} R^{(n)}$ of a $\Lambda$-coalescent $\Pi$, but instead the time-changed, scaled limit $\lim_{n\to\infty}\frac{1}{n}\eren_{tn^{1-\alpha}}$. To our best knowledge, there does not seem to be a direct link between these two results.\\}
Note that if we see the Bolthausen-Sznitman $n$-coalescent as a $Beta(1,1)$-coalescent and Kingman's $n$-coalescent as the borderline case of a Beta distribution with parameter ${\alpha}\to2$, the convergence results for $\sigmen$ shows a nice continuity in the parameters of the limit distributions in the range of $Beta(2-\alpha,\alpha)$-$n$-coalescents with $\alpha\in[1,2]$. Our convergence result itself is even somewhat true in the border cases 1 and 2 (if one wages $(\alpha-1)^{-1}\to\infty$ for $\alpha\to 1$ against $\log(n)\to\infty$ for $n\to\infty$ for the Bolthausen-Sznitman $n$-coalescents). Also note that $T$ obtained as the limit variable of $nT^{(n)}$ in Kingman's case has the same law as $2((1-\sigma)^{-1}-1)$, which  gives again a nice continuity in results. The continuity also appears for the block counting process: replacing ${\alpha}$ by 2 in the formula for the $Beta(2-\alpha,\alpha)$-$n$-coalescent gives the formula for Kingman's, this result will be proved in the sequel.

Finally we can remark that together with the known asymptotics $\frac{\ton}{n}\stackrel{d}{\to} \alpha-1$ for this class of $n$-coalescents (see \cite{DDS2008}, \cite{MR2365877} and \cite{iksanov2008number}), we have $\sigmen/\ton\stackrel{d}{\to}\sigma$.     

To prove these results, we will exploit some techniques from \cite{DDS2008}. In \cite{DDS2008}, they were used to analyze the asymptotics of a part of the length of a $n$-coalescent and the number of collisions in a $n$-coalescent for the same class of $n$-coalescents as analyzed in the present paper. For the convergence result for $\teen$, we present two proofs. One mimics the approach in \cite{caliebe2007length} and \cite{Janson11theexternal}, using the convergence result for $\sigmen$ and $\teen=\sum^{\sigmen}_{i=1}T_i$, where $T_i$ is the waiting time between the $(i-1)$th and $i$th collision/jump of the $n$-coalescent. The other proof is based on the representation of $T^{(n)}$ as the first jump time of a Cox process driven by a random rate process which depends only on the block counting process associated with the remaining individuals labelled $\{2,3,\ldots,n\}$. We use a recursive construction suitable for any $\Lambda$-$n$-coalescent: This construction consists in adding individual $i$ to a coalescent process constructed by individuals from $1$ to $n$ except $i$ such that consistence relationship is fulfilled.

 \subsection{Organization of the paper} In section $2$, we recall some known technical results which can all be found in \cite{DDS2008}. In section $3$, we obtain the asymptotic result about $\sigma^{(n)}$ and also about the ratio between $\sigma^{(n)}$ and $\tau^{(n)}$. Section $4$ studies the small time behavior of the block counting process $R^{(n)}$. Depending on the property of $R^{(n)}$, our first method taking $T^{(n)}$ as the first jump time of a Cox process gives the asymptotic behavior of $T^{(n)}$ in section $5$. In section $6$, another method is provided by taking into account the fact that $T^{(n)}$ is the sum of $\sigma^{(n)}$ initial waiting times for the coalescent process $\Pi^{(n)}$ to jump from one state to the following.

\section{Preliminaries}
In this Section, we recall some results from \cite{DDS2008}.

Consider a $n$-coalescent with multiple collisions characterized by a finite measure $\Lambda$ on $[0,1]$.
Let $\nu(dx)=x^{-2}\Lambda(dx)$ and $\rho(t)=\nu[t,1]$.
When the process has $k$ blocks, the next coalescence event comes at rate $g_k$ given by

\begin{equation}
   \label{eq:rgk}
g_k=\sum_{\ell=1}^{k -1} \binom{k }{\ell+1}
\lambda_{k ,\ell+1}=\int_{(0,1)} \Big(1- 
(1-x)^k  -k  x(1-x)^{k -1} \Big) \frac{\Lambda(dx)}{x^2}.
\end{equation}

For $n\geq 1$, $x\in (0,1)$, let $B_{n,x}$ be a binomial r.v. with
parameter $(n,x)$. Recall that for $1\leq k\leq n$, we have
\begin{equation}
   \label{eq:rpBnx}
\P(B_{n,x}\geq  k)=\frac{n!}{(k-1)!(n-k)!} \int_0^x t^{k-1} (1-t)^{n-k} \; dt.
\end{equation}
Use the first equality in \reff{eq:rgk} and
\reff{eq:rpBnx} to get 
\begin{align*}
g_n
\nonumber
&=\int_0^1 \sum_{k=2}^n \binom{n}{k} x^k (1-x)^{n-k} \nu(dx)\\
\nonumber
&=\int_0^1 \P(B_{n,x} \geq 2) \nu(dx)\\
&=n(n-1) \int_0^1 (1-t)^{n-2} t
\rho(t) \; dt. 
\end{align*}

All along this paper, the following hypothesis will be assumed
 \begin{equation}\label{eq:hyp}
 \rho(t)=C_0t^{-\alpha}  +O(t^{-\alpha+\zeta})
 \end{equation}
  for  some 
  $C_0>0$, $\alpha\in(1,2)$ and $\zeta>1-1/\alpha$.
  Under this hypothesis, Lemma 2.2 of \cite{DDS2008} gives us that, for $n\geq2$,

\begin{equation}\label{eq:rdlg}
 g_n = C_0 \Gamma(2-\alpha) n^\alpha +  O(n^{\alpha-\min(\zeta,1)}).
\end{equation}

Recall that we call $\ton$ the number of coalescence events until reaching the common ancestor of the initial population (of size $n$).
For $k\geq0$, denote by $\igreken_k$ the number of blocks remaining after $k$ jumps. Notice that $\igreken$ is a decreasing Markov chain with
$\igreken_0=n$ and 
$\igreken_k=1$ for $k\geq\ton$.
Let $\ixen_k= \igreken_{k-1}-\igreken_k $ be the number of blocks we lose during the $k$th coalescence event.
We write $\ixen_0=0$.

The Markov property makes that the law of the first jump $\ixen_1$ will be of much interest. {We will look at some properties of $\ixen_1$.}
Notice that 
\begin{equation}
\label{eq:lawX1}
\P(X^{(n)}_1=k)=\inv{g_n} \int_0^1
\P(B_{n,x}=k+1) \nu(dx)
\end{equation} 
and thus
\begin{equation}
   \label{eq:rbfxn}
\P(X^{(n)}_1\geq k)=\frac{\int_0^1 \P(B_{n,x} \geq k+1)
  \nu(dx)}{g_n}=\frac{(n-2)!}{k!(n-k-1)!} \frac{\int_0^1
(1-t)^{n-k-1} t^k \rho(t) \; dt} {\int_0^1
(1-t)^{n-2} t \rho(t) \; dt}  .
\end{equation}

Under the same assumptions on $\rho(t)$, setting $\varepsilon_0>0$ and

\begin{equation}
   \label{eq:rfn}
\varphi_n=\begin{cases}
{n^{-\zeta} } &\quad \text{if}\quad \zeta<\alpha-1,\\
 {n^{1-\alpha+\varepsilon_0}} &\quad \text{if}\quad \zeta=\alpha-1,\\
 n^{1-\alpha} &\quad \text{if}\quad \zeta>\alpha-1,
\end{cases}
\end{equation}
Lemma 2.3 of \cite{DDS2008} tells us there exists a
  constant $C_{\ref{eq:rM1}}$ s.t. for all $n\geq 2$, we have
\begin{equation}
   \label{eq:rM1}
\val{\E[X^{(n)}_1] - \inv{{\alpha-1}} } \leq
C_{\ref{eq:rM1}}\varphi_n.
\end{equation}
\\
Moreover, from Lemma 2.4 of \cite{DDS2008}, there exists a
  constant $C_{\ref{eq:rM2}}$ s.t. for all $n\geq 2$, we have

\begin{equation}
   \label{eq:rM2}
\E\left[\left(X^{(n)}_1\right)^2\right] \leq  C_{\ref{eq:rM2}} \frac{n^2}{g_n}.
\end{equation}

We consider $\phi_n$ the Laplace transform of $X_1^{(n)}$: for $u\geq 0$, 
$   \phi_n(u)=\E[\expp{-u X^{(n)}_1}]$.
   Assume  that hypothesis (\ref{eq:hyp}) holds true. Let $\varepsilon_0>0$. Recall $\varphi_n$ given
  by \reff{eq:rfn}. 
Then we have (see \cite{DDS2008}, Lemma 2.5)
, for $n\geq 2$, 
\begin{equation}
   \label{eq:rphin}
\phi_n(u)=1 -\frac{u}{{\alpha-1}} + \frac{u^\alpha}{{\alpha-1}} + R(n,u) ,
\end{equation}
where $\displaystyle R(n,u)= \left(u
  \varphi_n + u^2 \right)h(n,u)$ with $\sup_{u\in [0,K], n\geq 2}
|h(n,u)|<\infty $ for all $K>0$.  

Moreover, if we assume that $\zeta>1-1/\alpha$ and set ${\eta}\geq \inv{\alpha}$, then (from \cite{DDS2008}, Lemma 3.2) there exist ${\varepsilon_1}>0$ and
   $C_{\ref{eq:rupperboundR}}(K)$ a finite constant  such that for all
    $n\geq 1$ and $u\in [0,K]$, a.s. with $a_n=n^{-\eta}$, 
\begin{equation}\label{eq:rupperboundR}
\sum_{i=1}^{\tau_n } \left|R(Y_{i-1}^{(n)}, ua_n)\right|\leq
C_{\ref{eq:rupperboundR}}(K)n^{-\varepsilon_1}. 
\end{equation}

We will also use the following result :
Let $V=(V_t,t\geq 0)$ be a $\alpha$-stable L\'evy process with no positive
jumps   (see  chap.    VII   in  \cite{MR1406564})   with  Laplace   exponent
$\psi(u)=u^\alpha/(\alpha-1)$: for all $u\geq 0$, $\E[\expp{-uV_t}]=\expp{t
  u^\alpha/(\alpha-1)}$. 
We assume that $\rho(t)=C_0t^{-\alpha}  +O(t^{-\alpha+\zeta})$ for  some 
  $C_0>0$ and $\zeta>1-1/\alpha$. {Recall that $\ton$ is the number of coalescing events in the $n$-coalescent until reaching its absorbing state.} 
 Let 
$$
V^{(n)}_t=n^{-1/{\alpha}}\sum_{k=1}^{\lfloor nt\rfloor\wedge\ton}(X^{(n)}_k-\inv{\alpha-1})
$$
for $t\in [0,{\alpha-1})$, and 
$$V^{(n)}_{\alpha-1}=n^{-1/{\alpha}}\sum_{k=1}^{\ton}(X^{(n)}_k-\inv{\alpha-1})= n^{-1/\alpha}
\left(n-1-\frac{\ton}{\alpha-1} \right).$$
Then, 
\begin{equation}\label{eq:fdvt}
(V^{(n)}_t, t\in
[0,\alpha-1])\to(V_t, t\in
[0,\alpha-1])
\end{equation}
in the sense of convergence in law of the finite-dimensional marginals (see \cite{DDS2008}, Corollary 3.5,
see also \cite{MR2365877,iksanov2008number}).

\section{The number of collisions in an external branch}
\label{sec:sig}

Consider a $n$-coalescent with multiple collisions characterized by a finite measure $\Lambda$ on $[0,1]$.
Recall that  $\nu(dx)=x^{-2}\Lambda(dx)$ and $\rho(t)=\nu[t,1]$ which is assumed to satisfy (\ref{eq:hyp}).
{A $n$-}coalescent takes its values in $\cp^{(n)}${, the set of partitions of $\left\{1,\ldots,n\right\}$}.
For $i\geq0$, let $\pi_i=\pi^{(n)}_i$ be the state of the process after the $i$th coalescence event.

Pick at random an individual from the initial population and denote by $\teen$ the length of the external branch starting from it. 
Because of exchangeability, $\teen$ has the same law as the length $\teen_1$ of the external branch starting from the initial individual labelled by $\{1\}$.
{A quantity of interest will be $\sigmen$, the number of coalescence events we have to wait to see the randomly chosen external branch merging. Again because of exchangeability, $\sigmen$ has the same law as
$$
\sigmen_1=\inf\{i>0,\{1\}\notin\pi_i\},
$$
the time we have to wait to see the external branch linked to individual 1 merging. We can write
\begin{equation}\label{eq:defteen}
\teen_{1}=\sum_{i=1}^{\sigmen_1}\frac{e_i}{g_{\igreken_{i-1}}}
\end{equation}
where the $e_i$'s are i.i.d. exponential random variables with mean 1. Note that the formula also holds true for $\teen$ and $\sigmen$ (just omit the subscripts). For the remainder of the chapter, we will identify $\sigmen$ with $\sigmen_1$.}

In this section, we will determinate the asymptotic law of $\sigmen$ for a class of coalescents containing the $Beta$-coalescent with $\alpha\in(1,2)$.

\begin{theo}\label{th:cvsig}
We assume  that $\rho(t)=C_0t^{-\alpha}  +O(t^{-\alpha+\zeta})$ for  some 
  $C_0>0$, {$\alpha\in(1,2)$} and $\zeta>1-1/\alpha$.
Then
\begin{equation}
\frac{\sigmen}{n(\alpha-1)}
\overset{d}{\to}
\sigma,
\end{equation}
for $n\to\infty$, 
where $\sigma\stackrel{d}{=} Beta(1,\alpha)$.
\end{theo}
Recall that in this class of $n$-coalescents, we also have $\tau^{(n)}/n\stackrel{d}{\to}\alpha-1$ (from \cite{DDS2008}, see also \cite{MR2365877} and \cite{iksanov2008number}) for $n\to\infty$. Slutsky's theorem gives a convergence result for $\sigma^{(n)}/\tau^{(n)}$, which measures how ancient the chosen individual is compared to the rest of the sample 
\begin{corollary}
We assume  that $\rho(t)=C_0t^{-\alpha}  +O(t^{-\alpha+\zeta})$ for  some 
  $C_0>0$, {$\alpha\in(1,2)$} and $\zeta>1-1/\alpha$. Then
 $$\frac{\sigmen}{\ton}\stackrel{d}{\to}\sigma,$$
for $n\to\infty$, where $\sigma\stackrel{d}{=} Beta(1,\alpha)$.
\end{corollary}

\begin{proof}[proof of Theorem \ref{th:cvsig}]
For convenience, we set 

\begin{enumerate}
\item $\binom{a}{b}=0$, if $0\leq a<b, a\in\mathbb{Z_{+}},b\in\mathbb{Z_{+}},\mathbb{Z}_{+}=\{0,1,2,...\}.$
\item $\log(0)=-\infty$.
\end{enumerate}

Notice that $\sigmen\leq\ton$. Let $\cy=(\cy_k,  k\geq 0)$  denotes the  filtration generated  by  $\igreken$. 
For any $t\geq0$, we have

\begin{align*}
\P(\sigmen>nt|\cy)
&=
\P(\sigmen>nt,\ton>nt|\cy)\\
&=
\prod_{i=1}^{\lfloor nt\rfloor\wedge\ton}\P(\{1\}\in\pi_{i}|\{1\}\in\pi_{i-1},\cy)\\
&=
\prod_{i=1}^{\lfloor nt\rfloor\wedge\ton}\frac{\binom{\igreken_{i-1}-1}{\ixen_{i}+1}}{\binom{\igreken_{i-1}}{\ixen_{i}+1}}\\
&=
\prod_{i=1}^{\lfloor nt\rfloor\wedge\ton}\frac{\igreken_{i-1}-(\ixen_{i}+1)}{\igreken_{i-1}}.\\
\end{align*}

Notice that $\frac{\ixen_{i}+1}{\igreken_{i-1}}< 1$ if $i\neq\ton$, and $\frac{\ixen_{i}+1}{\igreken_{i-1}}=1$ if $i=\ton$.

We can hence write

\begin{align*}
\log\left(\P(\sigmen>nt|\cy)\right)
&=
\sum_{i=1}^{\lfloor nt\rfloor\wedge\ton}\log\left(1-\frac{\ixen_{i}+1}{\igreken_{i-1}}\right)\\
\end{align*}
and proceed to a power series expansion :
$$\log\left(\P(\sigmen>nt|\cy)\right)
=I^{(1)}_{nt}+I^{(2)}_{nt},
$$
with
$$I^{(1)}_{nt}=-\sum_{i=1}^{\lfloor nt\rfloor\wedge\ton}\left(\ixen_{i}+1\right)\left(\igreken_{i-1}\right)^{-1}
,$$
and
$$I^{(2)}_{nt}=\sum_{i=1}^{\lfloor nt\rfloor\wedge\ton}(\log(1-\frac{\ixen_{i}+1}{\igreken_{i-1}})+\frac{\ixen_{i}+1}{\igreken_{i-1}}),
$$

where $I^{(2)}_{nt}$ can be $-\infty$ if $\frac{\ixen_{i}+1}{\igreken_{i-1}}=1$. Let us look further at $I^{(1)}_{nt}$. 
The idea is to replace $\ixen_i$ by the limit of its expectation.
$$
I^{(1)}_{nt}=J^{(1)}_{nt}+J^{(2)}_{nt},$$
with
$$J^{(1)}_{nt}=-\sum_{i=1}^{\lfloor nt\rfloor\wedge\ton}\left(\inv{\alpha-1}+1\right)\left(\igreken_{i-1}\right)^{-1},
$$
and
$$J^{(2)}_{nt}=
-\sum_{i=1}^{\lfloor nt\rfloor\wedge\ton}\left(\ixen_{i}-\inv{\alpha-1}\right)\left(\igreken_{i-1}\right)^{-1}.
$$
We will use three lemmas whose proofs are given in the rest of the Section.

Lemma \ref{lem:calcul1}, with $\eta=1$, tells us that, when $0<t<\alpha-1$
\begin{equation}\label{eq:part1}
J^{(1)}_{nt}
\stackrel{\P}
\rightarrow -\frac{\alpha}{\alpha-1}\int_0^t\left(1-\frac{x}{{\alpha-1}}\right)^{-1}dx
=\alpha\log\left(1-\frac{t}{\alpha-1}\right).
\end{equation} 

Lemma \ref{lem:calcul2} gives, for $0<t<\alpha-1$,
\begin{equation}\label{eq:part2}
J^{(2)}_{nt}=-n^{1/\alpha-1}
n^{1-1/\alpha}\sum_{i=1}^{\lfloor nt\rfloor\wedge{\tau^{(n)}}}\left(\ixen_{i}-\inv{\alpha-1}\right)\left(\igreken_{i-1}\right)^{-1}
\xrightarrow{\P}0.
\end{equation}

Finally, Lemma \ref{lem:calcul3} gives, for $t<\alpha-1$,
\begin{equation}\label{eq:part3}
I_{nt}^{(2)}\xrightarrow{\P}0.
\end{equation}

Adding \eqref{eq:part1}, \eqref{eq:part2} and \eqref{eq:part3}, we get that {for $t<\alpha-1$}
$$
\log\left(\P(\sigmen>nt|\cy)\right)
\stackrel{\P}{\rightarrow}
\alpha\log\left(1-\frac{t}{\alpha-1}\right),
$$
and thus
$$
\P(\sigmen>nt|\cy)
\stackrel{\P}{\rightarrow}
\left(1-\frac{t}{\alpha-1}\right)^\alpha.
$$

While we know that $\P(\sigmen>nt|\cy)\leq 1$,then
$$\mathbb{P}(\sigmen>nt)=\mathbb{E}[\mathbb{P}(\sigmen>nt|\cy)]\rightarrow\left(1-\frac{t}{\alpha-1}\right)^\alpha.$$
We thus obtain that, for $x\in(0,1)$,
$$
\P(\sigmen>n(\alpha-1)x)\to(1-x)^\alpha.
$$
and then that $\displaystyle \frac{\sigmen}{n(\alpha-1)}$ converges in distribution to a $Beta(1,\alpha)$ law.
\end{proof}
\begin{lem}\label{lem:calcul1}
We set $\nu_{\eta}(t)=\int_0^t\left(1-\frac{x}{{\alpha-1}}\right)^{-\eta}dx$,$\eta\in \mathbb{R}$. 
We assume  that $\rho(t)=C_0t^{-\alpha}  +O(t^{-\alpha+\zeta})$ for  some 
  $C_0>0$, {$\alpha\in(1,2)$} and $\zeta>1-1/\alpha$.
For any $0<t<\alpha-1$ and $\eta\in \mathbb{R}$, we have 

\begin{enumerate}
\item Let $t_{0}\in [0,\alpha-1)$ and $\delta>0$. The following convergence in probability holds when $n\to \infty $:
$$n^{(\alpha-1)/2-\delta}\sup_{0 \leq t\leq t_{0}}|n^{\eta-1}\sum_{i=1}^{\lfloor nt\rfloor \wedge\ton}\left(\igreken_{i-1}\right)^{-\eta}-\nu_{\eta}(t)|\xrightarrow{}0.$$ 
\item Let $t\in [0,\alpha-1)$. The following convergence in distribution holds when $n\to \infty $:
$$n^{\eta-1/\alpha}(\sum_{i=1}^{\lfloor nt\rfloor \wedge\ton}\left(\igreken_{i-1}\right)^{-\eta}-n^{1-\eta}\nu_{\eta}(t))\xrightarrow{} \eta\int_{0}^{t}dr(1-\frac{r}{\gamma})^{-\eta-1}V_{r}.$$
\end{enumerate}
\end{lem}

\begin{proof}
The case ${\eta}={\alpha}-1$ is given by  Theorem $5.1$ in \cite{DDS2008}. Following the same arguments, it is easy to get the general result. 
\end{proof}

\begin{lem}\label{lem:calcul2}
For any $t<\alpha-1$, the following convergence in distribution holds :
$$n^{1-1/\alpha}\sum_{i=1}^{\lfloor nt\rfloor\wedge\ton}\left(\ixen_{i}-\inv{\alpha-1}\right)\left(\igreken_{i-1}\right)^{-1}
\xrightarrow{d}(v_\alpha(t))^{1/\alpha}V_1,$$
{where $(V_t)_{t\geq 0}$ is an $\alpha$-stable L\'evy process with no positive
jumps and  $v_\alpha(t)=\int_0^t\left(1-\frac{x}{{\alpha-1}}\right)^{-\alpha}dx$.}

\end{lem}
\begin{proof}
Let $\delta\in (0,\alpha-1)$, $t_0=\alpha-1-\delta$  and  $t\in
[0,t_0] $.

Let $\varepsilon\in (0,1-\frac{t}{\alpha-1})$ and $\beta=1-\frac{t}{\alpha-1}-\varepsilon>0$.
We have 
$$n^{1-1/\alpha}\sum_{i=1}^{\lfloor nt\rfloor\wedge\ton}\left(\ixen_{i}-\inv{\alpha-1}\right)\left(\igreken_{i-1}\right)^{-1}=A_{nt}+B_{nt},
$$
 with
$$A_{nt}=n^{1-1/\alpha}\sum_{i=1}^{\lfloor nt\rfloor\wedge\ton}\left(\ixen_{i}-\inv{\alpha-1}\right)\left(\igreken_{i-1}\right)^{-1}\ind_{\{\igreken_{i-1}\geq n\beta\}},$$
and
$$B_{nt}=n^{1-1/\alpha}\sum_{i=1}^{\lfloor nt\rfloor\wedge\ton}\left(\ixen_{i}-\inv{\alpha-1}\right)\left(\igreken_{i-1}\right)^{-1}\ind_{\{\igreken_{i-1}< n\beta\}}.$$
{We will show that $B_{nt}$ converges to 0 in probability and that $A_{nt}$ weakly converges to $(v_\alpha(t))^{1/\alpha}V_1$ as $n\to\infty$.}\\
\noindent
\textbf{Convergence of  $A_{nt}$}.
Let $Z_i^{(n)}=n\left(\igreken_{i-1}\right)^{-1}\ind_{\{\igreken_{i-1}\geq n\beta\}}$.
We have that $\sup_{n,i\geq1}Z_i^{(n)}\leq\beta^{-1}$ a.s..

By using (\ref{eq:rphin}), it is enough to prove that 
\[
{\mathbb E}[\exp(-uA_{nt})]\; \xrightarrow[n\rightarrow \infty
]{} \;   \expp { v_\alpha(t)   u^\alpha/(\alpha-1)},
\]
for any $u$ positive, where $\expp { v_\alpha(t)   u^\alpha/(\alpha-1)}$ is the Laplace transform of $(v_\alpha(t))^{1/\alpha}V_1$.

Taking $uZ_i^{(n)}$ as $Z_i^{(n)}$, we shall only consider the case $u=1$. 

Let us consider $n^{-1}\sum_{i=1}^{\lfloor nt\rfloor\wedge\ton}\left(Z_i^{(n)}\right)^{\alpha}=
n^{\alpha-1}\sum_{i=1}^{\lfloor nt\rfloor\wedge\ton}\left(\igreken_{i-1}\right)^{-\alpha}\ind_{\{\igreken_{i-1}\geq n\beta\}}.$
We have, because the process $(\igreken_{i},i\geq0)$ is decreasing, that
\begin{align*}
\P(n^{-1}\sum_{i=1}^{\lfloor nt\rfloor\wedge\ton}\left(Z_i^{(n)}\right)^{\alpha}\neq
n^{\alpha-1}\sum_{i=1}^{\lfloor nt\rfloor\wedge\ton}\left(\igreken_{i-1}\right)^{-\alpha})
&=\P(\exists i;\igreken_{i-1}< n\beta)\\
&\leq \P(\{\igreken_{(\lfloor nt\rfloor\wedge\ton)-1}< n\beta)\\
& \leq  \P(n^{-1} \sum_{j=1}^{(\lfloor nt\rfloor\wedge\ton)-1}(X^{(n)}_j- \frac{1}{{\alpha-1}}) \geq \varepsilon).
\end{align*}

Use \eqref{eq:fdvt} to
get that the right-hand side of the last inequality converges to $0$
as $n$ goes to infinity. 
Using also Lemma~\ref{lem:calcul1} with $\eta=\alpha$, we have that
$$n^{\alpha-1}\sum_{i=1}^{\lfloor nt\rfloor\wedge\ton}\left(\igreken_{i-1}\right)^{-\alpha}
\xrightarrow{\P}v_\alpha(t),$$
 as $n\to\infty$. We can thus deduce that
 \begin{equation}\label{eq:zalphacv}
 n^{-1}\sum_{i=1}^{\lfloor nt\rfloor\wedge\ton}\left(Z_i^{(n)}\right)^{\alpha}
 \xrightarrow{\P}v_\alpha(t),
 \end{equation}
as $n\to\infty$.

{For $a>0$}, we set 
$$
M^{{(a)}}_{n,k}
=
\exp\left(
\sum_{i=1}^{k}
\left(
-n^{-1/\alpha}{a}Z_i^{(n)}X^{(n)}_i
-\log\phi_{Y^{(n)}_{i-1}}(n^{-1/\alpha}{a}Z_i^{(n)})
\right)
\right).
$$
The process $(M^{{(a)}}_{n,k},k\geq 1)$ 
is a bounded martingale w.r.t. the filtration $\cy$. Notice that
$\E[M^{{(a)}}_{n,k}]=1$. As $X_{i}^{n}=0$ and $Z_{i}^{(n)}=0$ for $i>\tau^{(n)}$, we also have
$$
M^{{(a)}}_{n,k}
=
\exp\left(
\sum_{i=1}^{k\wedge\tau^{(n)}}
\left(
-n^{-1/\alpha}{a}Z_i^{(n)}X^{(n)}_i
-\log\phi_{Y^{(n)}_{i-1}}(n^{-1/\alpha}{a}Z_i^{(n)})
\right)
\right).
$$ 
Using {$R(n,u)$ defined in}
\reff{eq:rphin}, we get that : 
\begin{align*}
&M^{{(a)}}_{n,\lfloor nt\rfloor}\\
&=
\exp
{\left(
-\sum_{i=1}^{\lfloor nt\rfloor\wedge\ton}n^{-1/\alpha}{a}Z_i^{(n)}(X^{(n)}_i-\inv{\alpha-1})
-\sum_{i=1}^{\lfloor nt\rfloor\wedge\ton}\frac{n^{-1}\left({a}Z_i^{(n)}\right)^{\alpha}}{\alpha-1}
-\sum_{i=1}^{\lfloor nt\rfloor\wedge\ton}R(Y^{(n)}_{k-1},n^{-1/\alpha}{a}Z_i^{(n)})
\right)
}\\
&=
\exp
\left(-{a}A_{nt}\right)
\exp
\left(
-n^{-1}\sum_{i=1}^{\lfloor nt\rfloor\wedge\ton}\frac{\left({a}Z_i^{(n)}\right)^{\alpha}}{\alpha-1}
-\sum_{i=1}^{\lfloor nt\rfloor\wedge\ton}R(Y^{(n)}_{k-1},n^{-1/\alpha}{a}Z_i^{(n)})
\right).
\end{align*}
Let 
$$\Lambda_n=
-n^{-1}\sum_{i=1}^{\lfloor nt\rfloor\wedge\ton}\frac{\left(Z_i^{(n)}\right)^{\alpha}}{\alpha-1}
+\frac{v_\alpha(t)}{\alpha-1},$$
and write
 \[
{\mathbb E}\left[\exp
\left(
-A_{nt}\right)\right]=A_1+A_2,
\]
with 
$\displaystyle 
A_1
={\mathbb E}\left[\expp{
-A_{nt}}
\left(
1-
\expp{
\Lambda_n}
\right)
\right]
$ and 
$\displaystyle 
A_2
={\mathbb E}\left[\expp{
-A_{nt}}
\expp{
\Lambda_n}
\right]
$.

First of all, let  us prove that $A_1$ converges to 0  when $n$ tends to
$\infty $.   Recall that the  r.v $Z_i^{(n)}$ are uniformly  bounded by
$\beta^{-1}$ a.s.. Thanks to \eqref{eq:rupperboundR}, we have 
\[
\E[\expp{-2A_{nt} }]
=\E\left[M^{{(2)}}_{n,\lfloor nt\rfloor}
\exp
\left(
n^{-1}\sum_{i=1}^{\lfloor nt\rfloor\wedge\ton}\frac{\left(2Z_i^{(n)}\right)^{\alpha}}{\alpha-1}
+\sum_{i=1}^{\lfloor nt\rfloor\wedge\ton}R(Y^{(n)}_{k-1},2n^{-1/\alpha}Z_i^{(n)})
\right)\right]
\leq M,
\]
where $M$ is a finite constant which does not  depend on $n$. By Cauchy-Schwarz' inequality, we get that 
\begin{equation}\label{eq:A1}
(A_1)^2
\leq 
\left({\mathbb E}\left[\expp{-A_{nt}}
\left|1-\expp{\Lambda_n} \right|\right]\right)^2
\leq
 {\mathbb E}\left[\expp{-2A_{nt}}\right]
{\mathbb E}\left[\left(1-\expp{\Lambda_n}\right)^2\right]
\leq 
M{\mathbb E}\left[\left(1-\expp{\Lambda_n}\right)^2
\right].
\end{equation}
The quantity $\Lambda_n$ is bounded and goes to $0$ in
probability when $n$ goes to infinity {(see \eqref{eq:zalphacv})}. Therefore, the right-hand side of \reff{eq:A1} converges to 0. This implies that
$\lim_{n\rightarrow\infty } A_1=0$. 

Let us now consider the convergence of  $A_2$. Remark that 
\begin{equation*}
A_2={\mathbb E}
\left[
M^{{(1)}}_{n,\lfloor nt\rfloor} 
\exp\left(
\frac{v_\alpha(t)}{\alpha-1}
+\sum_{k=1}^{\lfloor nt\rfloor\wedge\ton}R(Y^{(n)}_{k-1},n^{-1/\alpha}Z_i^{(n)})\right)
\right].
\end{equation*}
Recall that ${\mathbb E}[M^{{(1)}}_{n,\lfloor nt\rfloor} ]=1$. Using \eqref{eq:rupperboundR},
we get 
$$
\exp{\left(-C_{\ref{eq:rupperboundR}}{(}\beta^{-1}{)}n^{-{\varepsilon_1}}
 +\frac{v_\alpha(t)}{\alpha-1}\right)}
\leq A_2\leq \exp{\left(C_{\ref{eq:rupperboundR}}{(}\beta^{-1}{)}n^{-{\varepsilon_1}}
 +\frac{v_\alpha(t)}{\alpha-1}\right)}.
$$
We get that $\lim_{n\to \infty }A_2=\expp{ v_\alpha(t) 
  /(\alpha-1)}$, which achieves the proof.

\noindent
\textbf{Convergence of  $B_{nt}$}.
 {Here, we will use a similar approach as the one we used on the first half of p.\pageref{eq:zalphacv}.} The process $(\igreken_{i},i\geq0)$ is decreasing.
So if for some $i\leq\lfloor nt\rfloor$, 
$\igreken_{i-1}< n\beta$,
then we have
$\igreken_{\lfloor nt\rfloor-1}< n\beta$.
Thus we get 
$B_{nt}=B_{nt}\ind_{\{\igreken_{(\lfloor nt\rfloor\wedge\tau^{(n)})-1}< n\beta\}}$.
Moreover,
\[
\{\igreken_{(\lfloor nt\rfloor\wedge\tau^{(n)})-1}< n\beta\}
\subset \{
n^{-1}\sum_{j=1}^{(\lfloor nt\rfloor\wedge\tau^{(n)})-1}(X^{(n)}_j- \frac{1}{{\alpha-1}}) \geq \varepsilon\}, 
\]
and then for any $\varepsilon'>0$ 
\begin{align*}
   \P( | B_{nt}|\geq \varepsilon')
&= \P(\ind_{\{\igreken_{(\lfloor nt\rfloor\wedge\tau^{(n)})-1}< n\beta\}} |B_{nt}|\geq \varepsilon' )\\
&\leq \P(\{\igreken_{(\lfloor nt\rfloor\wedge\tau^{(n)})-1}< n\beta)\\
& \leq  \P(n^{-1} \sum_{j=1}^{(\lfloor nt\rfloor\wedge\tau^{(n)})-1}(X^{(n)}_j- \frac{1}{{\alpha-1}}) \geq \varepsilon).
\end{align*}
Use \eqref{eq:fdvt} to
get that the right-hand side of the last inequality converges to $0$
as $n$ goes to infinity. 

\end{proof}

Now we deal with $I_{nt}^{(2)}$.
\begin{lem}\label{lem:calcul3}
We assume  that $\rho(t)=C_0t^{-\alpha}  +O(t^{-\alpha+\zeta})$ for  some 
  $C_0>0$ and $\zeta>1-1/\alpha$.
Then, for any $t<\alpha-1$, we have
$$
|I^{(2)}_{nt}|=\sum_{i=1}^{\lfloor nt\rfloor\wedge\ton}-(\log(1-\frac{\ixen_{i}+1}{\igreken_{i-1}})+\frac{\ixen_{i}+1}{\igreken_{i-1}})\xrightarrow{\mathbb{P}}0,$$
 when $n\to\infty$.

\end{lem}
\begin{proof}
Let $0 \leq t< \alpha-1$. First of all, remark that:

\begin{equation}\label{restlimit}
\frac{Y_{\lfloor nt \rfloor}^{(n)}}{n}\xrightarrow{\P}\frac{\alpha-1-t}{\alpha-1}.
\end{equation}
Indeed, 
$$\frac{Y_{\lfloor nt \rfloor}^{(n)}}{n}=\frac{n-(\lfloor nt \rfloor)/(\alpha-1)}{n}-\frac{\sum_{k=1}^{\lfloor nt \rfloor}(X_{k}^{(n)}-1/(\alpha-1))}{n},$$
and we conclude using (\ref{eq:fdvt}) and the convergence of $\mathbb{P}(\ton> \lfloor nt \rfloor)$ to 1. Let us write 
\begin{align*}
|I_{nt}^{(2)}|
=A_n+B_n,
\end{align*}
with 
\[
A_n=|I_{nt}^{(2)}|\ind_{\left\{Y_{\lfloor nt \rfloor}^{(n)}<(1-t/({\alpha}-1))n/2\right\}}
\]and
\[
B_n=|I_{nt}^{(2)}|\ind_{\left\{Y_{\lfloor nt \rfloor}^{(n)}\geq (1-t/({\alpha}-1))n/2\right\}}.
\]
The convergence (\ref{restlimit}) implies that $A_n$ tends to 0 in probability. To prove the convergence of $B_n$, let us first notice that for $a\in (0,1)$, there exists a constant $C(a)$ such that, if $B_{n,x}$ is a binomial r.v. with
parameter $(n,x)$, then 
\begin{equation}
\label{eq:lm3.6}
0<-\int_{0}^{1}\E\left[\ind_{\{ 2\leq  B_{n,x}\leq (1-a)n\} }(\log(1-\frac{B_{n,x}}{n})+\frac{B_{n,x}}{n})\right]\nu(dx) \leq C(a).
\end{equation}
Indeed, there exists a constant $C'(a)$ such that for $u\in (0,1-a)$, $0<-\ln(1-u)-u\leq C'(a)u^2$. Hence, 
\begin{align*}
0
&<
-\int_{0}^{1}\E\left[\ind_{ \{2\leq  B_{n,x}\leq (1-a)n\} }(\log(1-\frac{B_{n,x}}{n})+\frac{B_{n,x}}{n})\right]\nu(dx)\\
&\leq 
C'(a)\int_{0}^{1}\E\left[\ind_{ \{2\leq  B_{n,x}\leq (1-a)n\} }(\frac{B_{n,x}}{n})^{2}\right]\nu(dx)\\
&\leq 
C'(a)\int_{0}^{1}\E\left[(\frac{B_{n,x}}{n})^{2}\right]\nu(dx)\\
&\leq 2C'(a)\int_{0}^{1}\E\left[\frac{B_{n,x}(B_{n,x}-1)}{n^{2}}\right]\nu(dx)\\
&=2C'(a)\frac{\int_{0}^{1}n(n-1)x^{2}\nu(dx)}{n^{2}}\\
&\leq 2C'(a)\Lambda([0,1])=:C(a).
\end{align*}

Let us set $a=(1-t/({\alpha}-1))/2$. Hence $B_n=|I_{nt}^{(2)}|\ind_{\left\{Y_{\lfloor nt \rfloor}^{(n)}\geq an\right\}}$. Notice that if $n$ is large enough such that $an\geq 2$, then if $Y_{\lfloor nt \rfloor}^{(n)}\geq an$ we have $\ton>nt$. Moreover, if $Y_{\lfloor nt \rfloor}^{(n)}\geq an$, for  $i\leq nt$, we have  $Y_{i}^{(n)}\geq an\geq a \igreken_{i-1}$ and $\ixen_{i}=\igreken_{i-1}-\igreken_{i}\leq (1-a)\igreken_{i-1}<(1-a/2)\igreken_{i-1}$. Using (\ref{eq:lawX1}), (\ref{eq:lm3.6}) and (\ref{eq:rdlg}), we get  that 
\begin{align*}
&\mathbb{E}[B_n]\\
\leq
&
\sum_{i=1}^{\lfloor nt\rfloor}
\mathbb{E}[
\mathbb{E}[
-(\log(1-\frac{\ixen_{i}+1}{\igreken_{i-1}})+\frac{\ixen_{i}+1}{\igreken_{i-1}})
\ind_{\{1\leq \ixen_{i}\leq (1-a)Y_{i-1}^{(n)} \}}
\ind_{\{\igreken_{i-1}\geq an\}}
|Y_{i-1}^{(n)}]
]
\\
\leq
&\sum_{i=1}^{\lfloor nt\rfloor}
\mathbb{E}[
-\mathbb{E}[
\int_{0}^{1}
\ind_{\{2\leq B_{Y_{i-1}^{(n)},x}\leq (1-a/2)Y_{i-1}^{(n)} \}}
\ind_{\{\igreken_{i-1}\geq an\}}
\frac{1}{g_{Y_{i-1}^{(n)}}}(\log(1-\frac{B_{Y_{i-1}^{(n)},x}}{Y_{i-1}^{(n)}})+\frac{B_{Y_{i-1}^{(n)},x}}{Y_{i-1}^{(n)}})\nu(dx)
|Y_{i-1}^{(n)}]
]\\
\leq 
&
\frac{C(a/2)nt}{g_{an}}\to 0,
\end{align*}
when $n$ tends to $\infty$. This achieves the proof of the Lemma.
\end{proof}

\section{A result on small-time behavior of the block process} 
We now turn to the study of the length of an external branch picked at random, denoted by $\teen$.
For any integer $k$ between 1 and $\ton$, define $A^{(n)}_{k}$ as the time when the $k$th jump is achieved.
This variable can be expressed as a sum of $k$ independent exponential random variables.
More precisely, 
$$A^{(n)}_k=\sum_{i=1}^{k\wedge\tau^{(n)}}\frac{e_{i}}{g_{Y_{i-1}^{(n)}}},$$
where the $e_{i}$'s are independent standard exponential variables. 
Notice that $\teen=\aen_{\sigmen}$.
We will first study asymptotics of $\aen_k$. For this, we use a two-step approximation method close to Section 4 of \cite{DDS2008}.
Define first
$$\taen_k=\sum_{i=1}^{k\wedge\tau^{(n)}}\frac{1}{g_{Y_{i-1}^{(n)}}},$$
obtained replacing the $e_i$'s by their mean,and 
$$\haen_k=\inv{C_0\Gamma(2-\alpha)}\sum_{i=1}^{k\wedge\tau^{(n)}}(Y_{i-1}^{(n)})^{-\alpha},$$
obtained replacing $g_b$ by its equivalent in \eqref{eq:rdlg}.

\begin{prop}
\label{timerescaled}
We assume  that $\rho(t)=C_0t^{-\alpha}  +O(t^{-\alpha+\zeta})$ for  some 
  $C_0>0$ and $\zeta>1-1/\alpha$.
Then, for any $t<\alpha-1$, we have
 $$n^{\alpha-1}\aen_{\lfloor nt \rfloor} 
 \overset{\P}{\to}
 \inv{C_0\Gamma(2-\alpha)}((1-\frac{t}{\alpha-1})^{1-\alpha}-1),$$
  when $n\to\infty$.
\end{prop}

The proof is a straight consequence of Lemma \ref{lem:calcul1} with $\eta=\alpha$ and the following Lemmas \ref{lem:app1} and \ref{lem:app2}.

\begin{lem}
\label{lem:app1}
Under the assumptions of Proposition \ref{timerescaled}, we have
$$n^{\alpha-1}(\taen_{\lfloor nt \rfloor}-\haen_{\lfloor nt \rfloor})\stackrel{\P}{\rightarrow}0,$$
  when $n\to\infty$.
\end{lem}

\begin{proof}
Use \eqref{eq:rdlg} to get\[
\taen_{\lfloor nt \rfloor}-\haen_{\lfloor nt \rfloor}
=\sum_{i=1}^{\lfloor nt \rfloor\wedge\tau^{(n)} }
\left(Y^{(n)}_{i-1}\right)^{-\alpha} O
\left(\left(Y^{(n)}_{i-1}\right)^{-\min(\zeta,1)} \right).
\]
The result then follows from Lemma \ref{lem:calcul1} with $\eta=\alpha+\min(\zeta,1)$.
\end{proof}

\begin{lem}
\label{lem:app2}
Under the assumptions of Proposition \ref{timerescaled}, we have
$$n^{\alpha-1}(\aen_{\lfloor nt \rfloor}-\taen_{\lfloor nt \rfloor)})\stackrel{\P}{\rightarrow}0,$$
  when $n\to\infty$.
\end{lem}
\begin{proof}
Recall that $\cy=(\cy_k,  k\geq 0)$  denotes the  filtration generated  by  $Y$. 
Conditionally on $\cy$, the random variables $\displaystyle
   \frac {e_i-1}{g_{Y^{(n)}_{i-1} }}$ are independent with zero mean. We
   deduce that 
\begin{align*}
    \E\left[\sup_{t\geq 0} (n^{\alpha-1}(\aen_{\lfloor nt \rfloor}-\taen_{\lfloor nt \rfloor}))^2 |\cy\right]
 &=n^{2\alpha-2}
 \E\left[\sup_{t\geq 0} \left(\sum_{i=1}^{\lfloor nt \rfloor\wedge\tau^{(n)}}
 \frac{e_i-1 }{g_{Y^{(n)}_{i-1} }}\right)^2 |\cy\right]\\
&\leq 4n^{2\alpha-2}
\sum_{i=1}^{\lfloor nt \rfloor\wedge\tau^{(n)}} \left(\frac{1
}{g_{Y^{(n)}_{i-1} }}\right)^2,
\end{align*}
where we used Doob's  inequality for
the  inequality. Thanks to \reff{eq:rdlg} and Lemma \ref{lem:calcul1} with $\eta=2\alpha$, we get the $4n^{2\alpha-2}
\sum_{i=1}^{\lfloor nt \rfloor\wedge\tau^{(n)}} \left(\frac{1
}{g_{Y^{(n)}_{i-1} }}\right)^2$ converges to $0$ in probability.

\end{proof}

Heuristically, combining Theorem \ref{th:cvsig} and Proposition \ref{timerescaled}, we should get that $n^{\alpha-1}\teen=n^{\alpha-1}\aen_{\sigmen}$
converges in law to 
$\inv{C_0\Gamma(2-\alpha)}((1-\sigma)^{1-\alpha}-1).$
This line of proof will be followed in the last section. However, in the next section we will first present another way to prove this result with a method based on the consistency property of exchangeable coalescents.
As a first step to this approach, we end this session with a result about small-time behavior of the block-counting process.

Let $\eren_t$ denote the number of blocks of the $n$-coalescent $\Pi^{(n)}$ at time $t$. The initial value $\eren_0$ is $n$. 
We show that the limit law of the process $\eren$ is deterministic under a certain time rescaling

\begin{theo}
\label{th:R}
We assume  that $\rho(t)=C_0t^{-\alpha}  +O(t^{-\alpha+\zeta})$ for  some 
  $C_0>0$ and $\zeta>1-1/\alpha$.
For any $t_{0}>0,\varepsilon>0$ , we have 
 \begin{equation}\label{eq:cvR}
 \P(\sup_{0\leq t \leq t_{0}}|n^{-1}\eren_{tn^{1-\alpha}}-(1+C_0\Gamma(2-\alpha)t)^{-1/(\alpha-1)}|>\varepsilon)\to0,
 \end{equation}
 when $n\to\infty$. 
\end{theo}

\begin{proof}
Let $0<r<\alpha-1$, we have the following relation :
\begin{align*}
\eren_{\aen_{\lfloor nr \rfloor}}=\igreken_{\lfloor nr \rfloor}&=
n-\sum_{j=1}^{\lfloor nr \rfloor\wedge \ton} X_{j}^{(n)}
\end{align*}

Let $t\in[0,t_{0}]$, and define
\begin{equation}\label{eq:defrt}
r(t)=(\alpha-1)(1-(1+C_0\Gamma(2-\alpha)t)^{-1/(\alpha-1)}),
\end{equation}
on $[0,t_{0}]$. Notice that
 $$\inv{C_0\Gamma(2-\alpha)}((1-\frac{r(t)}{\alpha-1})^{1-\alpha}-1)=t.$$
  Then thanks to Proposition \ref{timerescaled}, $n^{\alpha-1}\aen_{\lfloor nr(t) \rfloor}$ converges in probability to $t$.

Using the remark at the beginning of the proof in Lemma \ref{lem:calcul3}, we get the convergence

$$
n^{-1}\eren_{\aen_{\lfloor nr(t) \rfloor}}=\frac{Y_{\lfloor nr(t)\rfloor}^{(n)}}{n}
\overset{\P}{\to} (1-\frac{r(t)}{\alpha-1})=(1+C_0\Gamma(2-\alpha){t})^{-1/(\alpha-1)},
$$
when $n\to\infty$.
Moreover, since $R_{t}^{(n)}$ is decreasing,  then for any $0<\delta <1,$
$$\lim_{n\rightarrow \infty}\mathbb{P}(\eren_{\aen_{\lfloor nr(t-\delta t) \rfloor}}\leq \eren_{tn^{1-\alpha+1}}\leq \eren_{\aen_{\lfloor nr(t+\delta t) \rfloor}})=1.$$

The constant $\delta $ being arbitrary, we thus obtain the convergence in probability of $n^{-1} \eren_{tn^{1-\alpha}} $ to $(1+C_0\Gamma(2-\alpha)t)^{-1/(\alpha-1)}$. 

We obtain \eqref{eq:cvR} using again the fact that $R^{(n)}$ is a decreasing process.

\end{proof}

In fact, the asymptotic result concerning block counting process of Kingman's coalescent is also valid. The method is almost identical to that employed in the above Theorem.  In the context of Kingman's coalescent, we use the same notations $\Pi^{(n)}$, $A_{i}^{(n)}$, $R^{(n)}$. 
\begin{theo}
In the setting of the Kingman's coalescent, 
for any $t_{0}>0,\varepsilon>0$ , we have 
 \begin{equation}\label{eq:cvRKingman}
 \P(\sup_{0\leq t \leq t_{0}}|n^{-1}\eren_{tn^{-1}}-(1+t/2)^{-1}|>\varepsilon)\to0
 \end{equation}
 when $n\to\infty$. 
\end{theo}
Remark that this Theorem shows a nice continuity from Beta-coalescent(the process that we consider is more general but contains Beta-coalescent) to Kingman's coalescent., 
setting $\alpha=2$ in Theorem \ref{th:R}.
\begin{proof}

Recall that $A_{i}^{(n)}$ is the time when $i$th jump is achieved. When $\Pi^{(n)}$ has $b$ individuals at some time $t$, then the process encounters the following coalescence at rate ${b \choose 2}$ where two randomly chosen individuals will be coalesced. $\Pi^{(n)}$ remains $1$ when all individuals are coalesced. 

For $0<t<1$, we have
$$A_{\lfloor nt \rfloor}^{(n)}=\sum_{k=n-\lfloor nt \rfloor +1}^{n}\frac{e_{k}}{\binom{k}{2}}$$
where $e_{i}$s are i.i.d unit exponential variables. Notice that 
$$\mathbb{E}[nA_{\lfloor nt \rfloor}^{(n)}]=\sum_{k=n-\lfloor nt \rfloor+1}^{n}\frac{n}{{k \choose 2}}=2(\frac{1}{n-\lfloor nt \rfloor}-\frac{1}{n})n\rightarrow 2(\frac{1}{1-t}-1),$$

as $n$ tends to $\infty$. There exist a constant $K>0$, such that,
$${\text{Var}}(nA_{\lfloor nt \rfloor}^{(n)})=\sum_{k=n-\lfloor nt \rfloor+1}^{n}n^{2}(\frac{1}{{k \choose 2}})^{2}\leq \frac{K}{n}.$$

So we deduce that 
$$nA_{\lfloor nt \rfloor}^{(n)}\stackrel{L^{2}}{\rightarrow}2(\frac{1}{1-t}-1)=\frac{2t}{1-t}:=f(t)$$
as $n$ converges to $\infty$.

We denote by $f^{-1}(t):=t/(t+2)$ the inverse function of $f(t)$. 

Similarly, $R^{(n)}$ is decreasing,  so 
$$\mathbb{P}(R_{A_{\lfloor nf^{-1}(t-\delta)^{(n)}\rfloor}}^{(n)}\leq R_{tn^{-1}}^{(n)} \leq R^{(n)}_{A_{\lfloor nf^{-1}(t+\delta)^{(n)}\rfloor}})\rightarrow 1,$$

as $n$ tends to $\infty$ for any $0<\delta<t$.

So $\frac{R_{tn^{-1}}^{(n)}}{n}-\frac{R_{A_{\lfloor nf^{-1}(t)\rfloor}^{(n)}}^{(n)}}{n}\stackrel{d}{\rightarrow}0.$

Furthermore,
$$\frac{R_{A_{nf^{-1}(t)}^{^{(n)}}}^{(n)}}{n}=\frac{Y_{\lfloor nf^{-1}(t)\rfloor}^{(n)}}{n}=\frac{n-\lfloor nf^{-1}(t)\rfloor}{n}\rightarrow 1-f^{-1}(t)=\frac{1}{1+t/2},$$
as $n$ tends to $\infty$.
So $\frac{R_{tn^{-1}}^{(n)}}{n}\stackrel{d}{\rightarrow} \frac{1}{1+t/2}$.

Using again the decreasing property of $R^{(n)}_{t}$, we finish the proof.
\end{proof}


\section{The length of an external branch picked at random}
Dynamics of  any exchangeable coalescent with multiple mergers are characterized by rates $\lambda_{b,k}$ which suit a consistent relationship (this is Pitman's structure theorem, see \cite{MR1742892}, Lemma 18):
\begin{equation}\label{eq:pre}
\lambda_{b,k}=\lambda_{b+1,k+1}+\lambda_{b+1,k}.
\end{equation}
This relationship comes from the fact that $k$ given merging blocks among $b$ can coalesce in two ways while revealing an extra block : either the coalescence event implies the extra block (and then $k+1$ blocks will merge) or not.
Thus we get a recursive construction of the $n$-coalescent process $\Pi^{(n)}$. 

Let us define $\Pi^{(n,2)}$ as the coalescent process of individuals labelled from $2$ to $n$. 
Now we consider the individual labelled by $1$. The lineage of this individual  can be 'connected' to $\Pi^{(n,2)}$ 
\begin{itemize}
\item either at any of its jump times, in which case block $\{1\}$ participates to a multiple merger implying at least 3 blocks, and we call this collision ``Type 1'' (see Figure~\ref{fig:graph1}), 
\item or  at any other time to one of the present blocks and then participates to a binary collision, and we call it ``Type 2'' (see Figure~\ref{fig:graph2}).
\end{itemize}

\begin{figure}[h]
\begin{center}
\begin{minipage}{\linewidth}
\centering 
\includegraphics[scale=0.5]{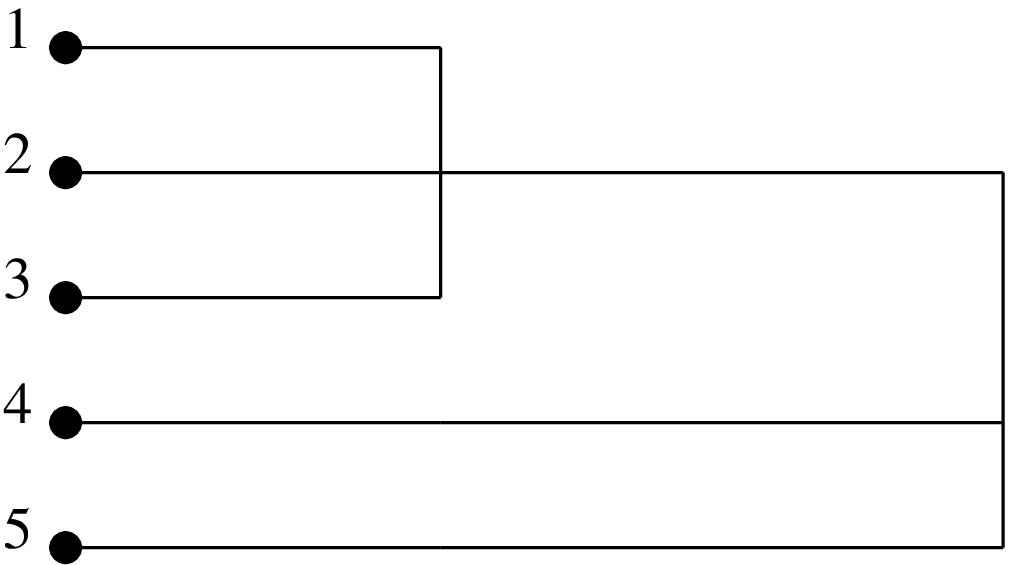}
\caption{$n=5$. Individual $1$ is chosen. Type $1$: individual $1$ encounters a multiple collision.} \label{fig:graph1}
\end{minipage}
\hfill
\begin{minipage}{\linewidth}
\centering 
\includegraphics[scale=0.5]{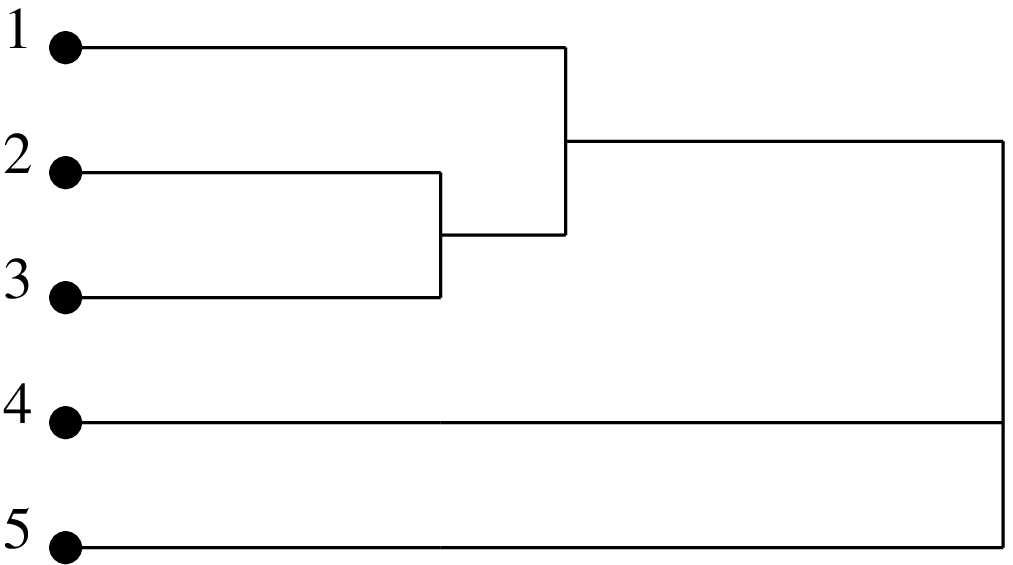}
\caption{$n=5$. Individual $1$ is chosen. Type $2$: individual $1$ encounters a binary collision} \label{fig:graph2}
\end{minipage}
\end{center}
\end{figure}

From now on, our analysis is conditional on $\Pi^{(n,2)}$.  Between two jump times of $\Pi^{(n,2)}$, assuming that there are $b$ blocks in $\Pi^{(n,2)}$, the extra block coalesces at rate $b\lambda_{b+1,2}$. 
If the extra block remains unconnected just before a coalescence event implying $k$ blocks among $b$, then it will participate to this event with probability 
\begin{equation}\label{eq:pourmail}
\frac{\int_{0}^{1}x^{k+1}(1-x)^{b-k}\nu(dx)/g_{b}}{\int_{0}^{1}x^{k}(1-x)^{b-k}\nu(dx)/g_{b}}
=1-\frac{\lambda_{b+1,k}}{\lambda_{b,k}}.
\end{equation}
This equality comes from \eqref{eq:pre}. Let us see how to get the law of $T^{(n)}$, the coalescence time of individual $1$.  We define by $R^{(n,2)}$  the block counting process of $\Pi^{(n,2)}$. Notice that it has the same law as $R^{(n-1)}$.  We introduce 
\begin{itemize}
\item $T_{c}^{(n)}$ the first jump time of a Poisson process $\eta_{c}^{(n)}$ directed by the measure  $\nu^{(n)}_{c}=R_{t}^{(n,2)}\lambda_{R_{t}^{(n,2)}+1,2}dt$;
\item $T_{d}^{(n)}$ the time of the first appearance of 'Head' in the following coin flip, independent of $\eta_{c}^{(n)}$ : at each jump time $t$ of  $R_{t}^{(n,2)}$, we toss a coin, and get 'Head' with probability 
$1-\frac{\lambda_{R_{t-}^{(n,2)}+1,R_{t-}^{(n,2)}-R_{t}^{(n,2)}+1}}{\lambda_{R_{t-}^{(n,2)},R_{t-}^{(n,2)}-R_{t}^{(n,2)}+1}}$ 
and 'Tail' with probability 
$\frac{\lambda_{R_{t-}^{(n,2)}+1,R_{t-}^{(n,2)}-R_{t}^{(n,2)}+1}}{\lambda_{R_{t-}^{(n,2)},R_{t-}^{(n,2)}-R_{t}^{(n,2)}+1}}$ (see (\ref{eq:pourmail})).
\end{itemize}
Then, conditionally on $\Pi^{(n,2)}$, $T^{(n)}$ and $T_{c}^{(n)}\wedge T_{d}^{(n)}$ have the same law.

\begin{rem}
A more formal way to interpret $T^{(n)}$ is as follow. Let $\xi^{(n)}$ be Cox process directed by random measure $\nu_{c}^{(n)}+\nu_{d}^{(n)}$, where 
$\nu^{(n)}_{d}=
\sum_{\{t \text{ is a jump time\}}}\frac{\lambda_{R_{t-}^{(n,2)}+1,R_{t-}^{(n,2)}-R_{t}^{(n,2)}+1}}{\lambda_{R_{t-}^{(n,2)},R_{t-}^{(n,2)}-R_{t}^{(n,2)}+1}}\delta_{t}$, and $\delta_{t}$ is the Dirac measure in $t$ (see \cite[p.226]{MR1876169}). Then $T^{(n)}$ has the same law as the first jump time of $\xi^{(n)}$
\end{rem}
Let us now give our main result

\begin{theo}
\label{thm2}
The following convergence holds :
 $$n^{\alpha-1}\teen
 \overset{d}{\to}
 T=\inv{C_0\Gamma(2-\alpha)}((1-\sigma)^{1-\alpha}-1),
$$
for $n\to\infty$.
 The density function of $T$ is
$$f_{T}(t)=\frac{\alpha C_0\Gamma(2-\alpha)}{\alpha-1}(1+C_0\Gamma(2-\alpha){t})^{-\frac{\alpha}{\alpha-1}-1},\quad t\geq 0.$$
In particular, in the $Beta(2-\alpha,\alpha)$ case, the density is
$$f_{T}(t)=\frac{1}{(\alpha-1)\Gamma(\alpha)}(1+\frac{t}{\alpha\Gamma(\alpha)})^{-\frac{\alpha}{\alpha-1}-1},\quad t\geq 0.$$
\end{theo}

\begin{proof}
For the sake of simplicity, we will make the proof only in the $Beta(2-\alpha,\alpha)$ case. The proof can be extended to the more general case where \eqref{eq:hyp} is satisfied with the details omitted here.
In this special case, $C_0=(\alpha\Gamma(\alpha)\Gamma(2-\alpha))^{-1}$ and dynamics are given by
$$\lambda_{b,k}=\frac{B(k-\alpha,b-k+\alpha)}{B(\alpha,2-\alpha)},$$
where $B(a,b)$ is a Beta function of parameters $a$ and $b$.

Define ${r}^{(n,2)}_t$ as the number of jumps of the process $\Pi^{(n,2)}$ up to time $n^{1-{\alpha}}t$.
It is a straightforward consequence of Proposition \ref{timerescaled} that

\begin{equation}
\label{newr}
\frac{{r}^{(n,2)}_t}{n} \overset{\P}{\to}r(t),\quad n\to\infty
\end{equation}
for $t\to\infty$,
where $r(t)$ is defined in \eqref{eq:defrt}.

For $i\geq0$, in the process $\Pi^{(n,2)}$, we denote by $Y_{i}^{(n,2)}$ the number of blocks remaining after $i$ jumps which equals $1$ from the time all individuals are coalesced to $1$, and $Y_{0}^{(n,2)}=n-1$.
Let $X_{i}^{(n,2)}= Y_{i-1}^{(n,2)}-Y_{i}^{(n,2)}$ be the number of blocks we lose during the $i$th coalescent event.
We write $\ixen_0=0$.
Notice that $(Y^{(n,2)},X^{(n,2)})$ has the same law as $((Y^{(n-1)},X^{(n-1)}))$ 
.

Using the description  given above, we have
\begin{eqnarray*}
&&\P(n^{\alpha-1}T^{(n)}>t)\\
&=&\mathbb{E}[\mathbb{P}(n^{\alpha-1}(T_{c}^{(n)}\wedge T_{d}^{(n)})>t|\Pi^{(n,2)})]\\
&=&\mathbb{E}[\mathbb{P}(n^{\alpha-1}T_{c}^{(n)}>t|\Pi^{(n,2)})\mathbb{P}(n^{\alpha-1}T_{d}^{(n)}>t|\Pi^{(n,2)})]\\
&=&\mathbb{E}[\exp(-\int_{0}^{t}\int_{0}^{1}n^{1-\alpha}R_{sn^{1-\alpha}}^{(n,2)}x^{2}(1-x)^{R_{sn^{1-\alpha}}^{(n,2)}-1}\nu(dx)ds)
\prod_{i=1}^{r^{(n)}_t}\frac{\lambda_{1+Y_{i-1}^{(n,2)},1+X_i^{(n,2)}}}{\lambda_{Y_{i-1}^{(n,2)},1+X_{i}^{(n,2)}}}]\\
&=&\mathbb{E}[\exp(-\int_{0}^{t}n^{1-\alpha}R_{sn^{1-\alpha}}^{(n,2)}\frac{B(2-\alpha,R_{sn^{1-\alpha}}^{(n,2)}+\alpha-1)}{B(2-\alpha,\alpha)}ds)
\prod_{i=1}^{{r}^{(n)}_t}\frac{Y_{i-1}^{(n,2)}-X_{i}^{(n,2)}+\alpha-1}{Y_{i-1}^{(n,2)}}].
\end{eqnarray*}

We decompose the term in the expectation into two parts: the exponential on one side and the product on the other.

Let us first look at the exponential term.
Using Stirling's formula we get that, for $0\leq s\leq t$,
$$
n^{1-\alpha}R_{sn^{1-\alpha}}^{(n,2)}\frac{B(2-\alpha,R_{sn^{1-\alpha}}^{(n,2)}+\alpha-1)}{B(2-\alpha,\alpha)}
=n^{1-\alpha}\frac{(R_{sn^{1-\alpha}}^{(n,2)})^{\alpha-1}}{\Gamma(\alpha)}+(\frac{R_{sn^{1-\alpha}}^{(n,2)}}{n})^{\alpha-1}f(R_{sn^{1-\alpha}}^{(n,2)}),
$$
where $f=f(t)_{\{t\geq0\}}$ is a deterministic function which converges to $0$ as $t$ converges to $\infty$.
The sequence $(R_{sn^{1-\alpha}}^{(n,2)},n\geq2)$ is decreasing so, thanks to Theorem \ref{th:R}, we deduce that $\sup_{0\leq s \leq t}(\frac{R_{sn^{1-\alpha}}^{(n,2)}}{n})^{\alpha-1}f(R_{sn^{1-\alpha}}^{(n,2)})
$
converges in probability to $0$ as $n$ tends to $\infty$. 
Consequently, using again Theorem \ref{th:R}, we get that 
\begin{equation}
\label{part1}
\exp(-\int_{0}^{t}n^{1-\alpha}R_{sn^{1-\alpha}}^{(n,2)}\frac{B(2-\alpha,R_{sn^{1-\alpha}}^{(n,2)}+\alpha-1)}{B(2-\alpha,\alpha)}ds)
\overset{\P}{\to}
(1+\frac{t}{\alpha\Gamma(\alpha)})^{-\alpha},\quad n\to\infty.
\end{equation}

Convergence of the product term is obtained by the same method as in proof of Theorem \ref{th:cvsig},
combined with the convergence in \eqref{newr}. To avoid showing almost the same reasoning, we leave the details to readers. This way, we have
\begin{equation}
\label{part2}
\prod_{i=1}^{{r}^{(n,2)}_t}\frac{Y_{i-1}^{(n,2)}-X_{i}^{(n,2)}+\alpha-1}{Y_{i-1}^{(n,2)}}
\overset{\P}{\to}
(1+\frac{t}{\alpha\Gamma(\alpha)})^{-\frac{\alpha(2-\alpha)}{\alpha-1}},\quad n\to\infty.
\end{equation}

The product of (\ref{part1}) and (\ref{part2}) then 
converges in probability to $(1+\frac{t}{\alpha\Gamma(\alpha)})^{-{\alpha}/(\alpha-1)}$.
Since 
this product is bounded, we get that
$$\P(n^{\alpha-1}\teen>t)
\overset{}{\to}
(1+\frac{t}{\alpha\Gamma(\alpha)})^{-\frac{\alpha}{\alpha-1}},\quad n\to\infty.
$$
We achieve the proof.
\end{proof}
As a consequence of Theorem~\ref{thm2}, we can get an asymptotic result on the size of the population at the moment of collision of individual $1$.

\begin{cor}
\label{cor:Y_sigma}
The following convergence holds :
\[
n^{-1}{Y_{\sigma^{(n)}}^{(n)}}
 \overset{d}{\to}
(1+C_0\Gamma(2-\alpha)T)^{-1/(\alpha-1)}=1-\sigma
\]
for $n\to\infty$.
Moreover, the density function of this limit is $\alpha x^{\alpha-1}\ind_{\{0\leq x \leq 1\}}$.
\end{cor}
\begin{proof}In terms of block counting process, we have $Y_{\sigma^{(n)}}^{(n)}=R_{T^{(n)}}^{(n)}$. 
Notice that ${R_{T^{(n)}}^{(n)}}={R_{n^{1-\alpha}(n^{\alpha-1}T^{(n)})}^{(n)}}$. Using Theorem \ref{thm2}, we known that $n^{\alpha-1}T^{(n)}$ converges in distribution to $T$. Hence, if $t_0>0$, we deduce from Theorem \ref{th:R} that 
\[
\ind_{\{n^{\alpha-1}T^{(n)}<t_0\}} \frac{R_{n^{1-\alpha}(n^{\alpha-1}T^{(n)})}^{(n)}}{n}
\overset{d}{\to}
\ind_{\{T<t_0\}}(1+C_0\Gamma(2-\alpha)T)^{-1/(\alpha-1)}.
\]
This achieves the proof.
\end{proof}

\section{An alternative proof for Theorem \ref{thm2}}
In this section, we present an alternative proof for Theorem \ref{thm2} using the convergence results for $\sigma^{(n)}$ from Theorem~\ref{th:cvsig}. First, we need a stronger version of Proposition \ref{timerescaled} which gives weak convergence in the path space.  Recall that  
$$A^{(n)}_k=\sum_{i=1}^{k\wedge\tau^{(n)}}\frac{e_{i}}{g_{Y_{i-1}^{(n)}}},$$
where the $e_{i}$'s are independent standard exponential variables. 
\begin{prop}
\label{timerescaled2}
We assume  that $\rho(t)=C_0t^{-\alpha}  +O(t^{-\alpha+\zeta})$ for  some 
  $C_0>0$ and $\zeta>1-1/\alpha$.
Then, for any $t<\alpha-1$, we have
\begin{equation}\label{aenpaths}
(n^{\alpha-1}\aen_{\lfloor ns\rfloor})_{s\leq t} 
 \overset{d}{\to}
 (\inv{C_0\Gamma(2-\alpha)}((1-\frac{s}{\alpha-1})^{1-\alpha}-1))_{s\leq t},
 \end{equation}
  in the sense of convergence in the path space $D[0,t]$ for $n\to\infty$.
\end{prop}
\begin{proof}
Note that Theorem \ref{timerescaled} states
$$n^{\alpha-1}\aen_{\lfloor ns\rfloor}\overset{\P}{\to}
 (\inv{C_0\Gamma(2-\alpha)}((1-\frac{s}{\alpha-1})^{1-\alpha}-1)),$$
for $0<s<\alpha-1$ and $n\to\infty$. 
So for every fixed $s\in[0,t]$, we have pointwise convergence in probability in (\ref{aenpaths}). This implies weak convergence of all finite dimensional distributions due to the subsequence criterion for weak convergence. In order to show weak convergence in the path space, we will show tightness for the distributions from (\ref{aenpaths}). Since the limit process is continuous, it suffices to show that the condition (i) of \cite[Theorem 7.3]{MR1700749} and condition (7.12) from \cite[Corollary 7.4]{MR1700749} are fulfilled (see \cite[Corollary 13.4]{MR1700749}). For the present processes, these conditions translate to showing that for every $\epsilon>0$ and $\eta>0$,
\begin{itemize}
\item[(i)] there exists $a>0$ s.t. $P(n^{\alpha-1}\aen_{\lfloor 0 \rfloor}\geq a)\leq \eta$ for $n$ big enough and 
\item[(ii)] there exists a $0<\delta<1$ so that $$\delta^{-1}P\left(n^{\alpha-1}(\aen_{\lfloor n\cdot\min(t_1+\delta,t)\rfloor}-\aen_{\lfloor n\cdot t_1 \rfloor})\geq \epsilon\right)\leq \eta,$$ for $n$ big enough and any $t_1\in[0,t]$.
\end{itemize}
Condition $(i)$ is trivially fulfilled, for condition (ii) we can use Theorem \ref{timerescaled} to show that for $n\to \infty$, 
$$P\left(n^{\alpha-1}(\aen_{\lfloor n\cdot\min(t_1+\delta,t)\rfloor}-\aen_{\lfloor n\cdot t_1\rfloor})\geq \epsilon\right)\to P(f(\min(t_1+\delta,t))-f(t_1)\geq \epsilon),$$  where $f(s):=\inv{C_0\Gamma(2-\alpha)}((1-\frac{s}{\alpha-1})^{1-\alpha}-1)$. Note that $P(f(\min(t_1+\delta,t))-f(t_1)\geq \epsilon)\leq P(f(t)-f(t-\delta)\geq \epsilon)\in\left\{0,1\right\}$. Since $f$ is continuous, you can now choose $\delta$ small enough that $f(t)-f(t-\delta)<\epsilon$ and then $n$ big enough to fulfill (ii). Thus, we have shown tightness of the distributions in (\ref{aenpaths}) which establishes the desired weak convergence   
\end{proof}
Now we come to the alternative proof of Theorem \ref{thm2}.
\begin{proof}[Alternative proof of Theorem \ref{thm2}]
Fix $t\in[0,\alpha-1)$. We have  
$$(\sigma^{(n)}/(n(\alpha-1)),(n^{\alpha-1}\aen_{\lfloor ns\rfloor})_{s\leq t})\stackrel{d}{\to}(\sigma,(\inv{C_0\Gamma(2-\alpha)}((1-\frac{s}{\alpha-1})^{1-\alpha}-1))_{s\leq t}),$$ 
for $n\to\infty$. Due to Skorohod-coupling, we can assume that this convergence also holds almost surely. Since $s\mapsto(\inv{C_0\Gamma(2-\alpha)}((1-\frac{s}{\alpha-1})^{1-\alpha}-1))$ is continuous on $[0,t]$, the almost sure convergence of $(n^{\alpha-1}\aen_{\lfloor nt \rfloor})_{s\leq t}$ in $D[0,t]$ is even almost sure uniform convergence on $[0,t]$ (see \cite[p. 124]{MR1700749}). For any series $(x_n)_{n\in\mathbb{N}}$ on [0,t] with $x_n\to x$, we thus have $$n^{\alpha-1}\aen_{\lfloor nx_n \rfloor}\to \inv{C_0\Gamma(2-\alpha)}((1-\frac{x}{\alpha-1})^{1-\alpha}-1)$$ almost surely for $n\to\infty$. The only problem left is that $\sigma^{(n)}$, $\sigma$ may take values in $[0,\alpha-1)$ and not only in some subset $[0,t]$. To remedy this, note that if we restrict all random variables on $\{ \sigma\leq \alpha-1-\frac{2}{k}\}$ for $k\in\mathbb{N}$, we have $\sigma^{(n)}(\omega)/n\leq \alpha-1-\frac{1}{k}$ for $n=n(\omega)$ big enough for almost all $\omega\in\{ \sigma\leq \alpha-1-\frac{1}{k}\}$. Thus, by using the Skorohod-coupling for the series $(\sigma^{(n)}/(n(\alpha-1)),(n^{\alpha-1}\aen_{\lfloor ns\rfloor})_{s\leq \alpha-1-\frac{1}{k}})$, we have
$$n^{\alpha-1}\aen_{\lfloor n\cdot\sigma^{(n)}/n\rfloor}\stackrel{d}{\to}\inv{C_0\Gamma(2-\alpha)}((1-\frac{\sigma}{\alpha-1})^{1-\alpha}-1),$$
almost surely on $\{\sigma\leq \alpha-1-\frac{2}{k}\}$ for the coupled versions of these random variables (note that $\sigmen\leq \ton$). Since $\sigma$ is Beta-distributed, we have, for $k\to\infty$,
$$P(\{n^{\alpha-1}\aen_{\lfloor n\cdot\sigma^{(n)}/n\rfloor}\in\cdot\}\cap \{\sigma\leq \alpha-1-\frac{2}{k}\})\sim P(n^{\alpha-1}\aen_{\lfloor n\cdot\sigma^{(n)}/n\rfloor}\in\cdot).$$
This shows 
$$n^{\alpha-1}T^{(n)}=n^{\alpha-1}\aen_{\sigma^{(n)}}\stackrel{d}{\to}\inv{C_0\Gamma(2-\alpha)}((1-\frac{\sigma}{\alpha-1})^{1-\alpha}-1).$$
The only thing left to prove is that $T$ has density $f_T$. This is done by computing the distribution function 
\begin{eqnarray*}
&&\mathbb{P}(\inv{C_0\Gamma(2-\alpha)}((1-\frac{\sigma}{\alpha-1})^{1-\alpha}-1)\leq t)\\
&=&\mathbb{P}(\sigma\leq 1-(1+C_{0}\Gamma(2-\alpha)t)^{\frac{1}{\alpha-1}})\\
&=&\int_{0}^{1-(1+C_{0}\Gamma(2-\alpha)t)^{\frac{1}{\alpha-1}}}\alpha(1-x)^{\alpha}dx\\
&=&1-(1+C_{0}\Gamma(2-\alpha)t)^{-\frac{\alpha}{\alpha-1}}.
\end{eqnarray*}
and finally by differentiating.  
\end{proof}

\bigskip
{\bf Acknowledgments.} 
Jean-St\'ephane Dhersin and Linglong Yuan  benefited from the support of the ``Agence Nationale de la Recherche'': ANR MANEGE (ANR-09-BLAN-0215). \\
Arno Siri-J\'egousse and Fabian Freund would like to thank the Haussdorf Research Institute for Mathematics for its support.\\
Fabian Freund would like to thank Luca Ferretti for a discussion concerning the interpretation of $\sigmen$.

\end{document}